\documentclass[12pt]{amsart}
\usepackage{amssymb,amsmath,amsthm}
\numberwithin{equation}{section}
\newtheorem{theorem}{Theorem}[section]
\newtheorem{lemma}[theorem]{Lemma}
\newtheorem{remark}[theorem]{Remark}

\newtheorem{definition}[theorem]{Definition}

\renewcommand{\(}{\left(}
\renewcommand{\)}{\right)}

\newcommand{\no}{\nonumber}

\newcommand{\si}{\sigma}
\allowdisplaybreaks

\begin{document}
\title[SEMICIRCULAR LAW OF RANDOM QUATERNION MATRICES]{On the semicircular law of large dimensional random quaternion matrices}

\author{YANQING YIN,\ \ ZHIDONG BAI, \ \ JIANG HU }
\thanks{Z. D. Bai was partially supported by CNSF  11171057 and PCSIRT; J. Hu was partially supported by a grant CNSF 11301063.
}

\address{KLASMOE and School of Mathematics \& Statistics, Northeast Normal University, Changchun, P.R.C., 130024.}
\email{yinyq799@nenu.edu.cn}
\address{KLASMOE and School of Mathematics \& Statistics, Northeast Normal University, Changchun, P.R.C., 130024.}
\email{baizd@nenu.edu.cn}
\address{KLASMOE and School of Mathematics \& Statistics, Northeast Normal University, Changchun, P.R.C., 130024.}
\email{huj156@nenu.edu.cn}


\maketitle
\begin{abstract}
It is well known that Gaussian symplectic ensemble (GSE)  is defined  on the space of $n\times n$ quaternion self-dual Hermitian matrices with Gaussian random elements. There is a huge body of literature regarding this kind of matrices based on the exact known form of density function of the eigenvalues (see \cite{ErdHos11U,ErdHosY11B,KnowlesY11E,TaoV11R,TaoV11W,ErdHosY12R}).  Due to the fact that multiplication of quaternions is not commutative, few works about large dimensional quaternion self-dual Hermitian  matrices are seen without normality assumptions.  As natural,  we shall  get  more  universal results  by  removing  the Gaussian condition.  For the first step, in this paper
 we prove that   the empirical  spectral distribution of the common quaternion self-dual Hermitian matrices tends to semicircular law.
 The main tool to establish the  universal result is  given as a lemma in this paper as well.

\end{abstract}

\section{introduction and main results}

Suppose  $H_n$ is an $n \times n$ Hermitian matrix with eigenvalues  ${s _i}, i = 1, 2,  \cdots n$.   The empirical spectral distribution (ESD) of  $H_n$ is defined as:
$${F^{H_n}}(x) = \frac{1}{n}\sum\limits_{i = 1}^n {I({s _i} \le x)} $$
where  ${I(\cdot})$ is  indicator function.
It is   shown  that if the entries   on and above the diagonal of $H_n$ (known as Wigner matrix) are independent random variables with zero-mean and
variance $n^{-1}\si^2$, then $F^{H_n}$
converges almost surely (a.s.) to  a non-random
distribution $F$ whose density function is given by 
\begin{align}
    f(x)=\frac{1}{2\pi\si^2}\sqrt{4\si^2-x^2},~~x\in[-2\sigma,2\sigma].\label{01}
\end{align}This is also known as the  semicircular law (see \cite{Wigner55C}).
Wigner matrices and semicircular law play  important roles in physics and pure mathematics. Thus in resent years, there are a lot of  subsequent work which were  trying to obtain a better understanding of Wigner matrices and  semicircular law. Much more details can be found in \cite{Mehta04R,TulinoV04R,BaiC09R,AndersonG10I,CouilletD11R} and references therein.

In mathematics, the quaternions  were first described by Irish mathematician William Rowan Hamilton in 1843 \cite{Hamilton44Q}, and applied to mechanics in three-dimensional space. In recent years, quaternions are found uses in both theoretical and applied mathematics, such as  in three-dimensional computer graphics and computer vision, the theory of
peg-top using, navigation, enginery and organ, robot technology and artificial satellite attitude control, and so on.
However,  random quaternion  matrices have not been studied as substantially as those in real and complex fields, except for Gaussian Symplectic Ensemble (GSE).

 The GSE is defined on the space of $n\times n$ quaternion self-dual Hermitian matrices. For a matrix $X = ({x_{jk}})_{n\times n}$ drawn from the GSE, where
$${x_{jk}}=a_{jk}+b_{jk}\cdot i_1+c_{jk}\cdot i_2+d_{jk}\cdot i_3$$ is a random quaternion.
 Here $\{  i_1 , i_2 , i_3 \}$ denotes the standard quaternion basis with
\begin{gather}
i_1^2=i_2^2=i_3^2=-1,\quad i_1=i_2i_3=-i_3i_2,\nonumber\\
 i_2=i_1i_3=-i_3i_1,\quad i_3=i_1i_2=-i_2i_1.\label{eqb1}
\end{gather}
The four coefficients $\{a_{jk},b_{jk},c_{jk},d_{jk}\}$  are  independent real Gaussian random variables with zero mean. For $j > k$,
\begin{align*}
    E a_{jk}^2=Eb_{jk}^2=Ec_{jk}^2=Ed_{jk}^2=1/4,
\end{align*}
  and for the diagonal elements $E{x^2_{jj}} = Ea^2_{jj} =1$.
It was shown that the ESD of GSE tends to semicircular law almost surely, and there are also many local and bulk results about GSE. All these results are derived based on the fact that the exact form of eigenvalues of GSE. Details can be found in \cite{ErdHos11U,ErdHosY11B,KnowlesY11E,TaoV11R,TaoV11W,ErdHosY12R,BaiH11N}.

 Therefore it motives us to investigate the  universal results of the quaternion self-dual Hermitian matrices under only necessary moment conditions. Before giving the main theorem, we introduce some notation  and basic properties of  quaternion. Define
\begin{displaymath}
\mathbf e = \left( \begin{array}{cc}
1&0\\
0&1\\
\end{array} \right), \quad
\mathbf{i} = \left( \begin{array}{cc}
i&0\\
0&- i\\
\end{array} \right), \quad
\mathbf j = \left( \begin{array}{cc}
0&1\\
-1&0\\
\end{array}\right), \quad
\mathbf{k} = \left( \begin{array}{cc}
0&i\\
i&0\\
\end{array}\right),
\end{displaymath}
where $i=\sqrt{-1}$.
It's easy to verify that
 \begin{gather}
\mathbf{i} ^2=\mathbf{j} ^2=\mathbf{k} ^2=-\mathbf{e},\quad \mathbf{i} =\mathbf{j} \mathbf{k} =-\mathbf{k} \mathbf{j} ,\nonumber\\
 \mathbf{j} =\mathbf{i} \mathbf{k} =-\mathbf{k} \mathbf{i} ,\quad \mathbf{k} =\mathbf{i} \mathbf{j} =-\mathbf{j} \mathbf{i} .\label{eqb2}
\end{gather}
The relations (\ref{eqb1}) and (\ref{eqb2}) establish an isomorphic relation between the quaternion defined by the linear combinations of $\{1,i_1,i_2,i_3\}$ and that by $\{{\bf e}, {\bf i,j,k}\}$.
 That is, a quaternion $$x=a+bi_1+ci_2+di_3$$can be represented as $$x = a \cdot \mathbf e + b \cdot \mathbf i + c \cdot \mathbf j + d \cdot \mathbf k = \left( {\begin{array}{*{20}{c}}
\lambda &\omega \\
{ - \overline \omega }&{\overline \lambda }
\end{array}} \right), $$
where $a, b, c, d$ are  real  and $\lambda  = a + bi$, $\omega  = c + di$ are complex. Here $i=\sqrt{-1}$ denotes the usual imaginary unit. The quaternion conjugate of $x$ is defined by
$$ \bar x = a \cdot \mathbf e - b \cdot \mathbf i - c \cdot \mathbf j - d \cdot \mathbf k = \left( {\begin{array}{*{20}{c}}
{\bar \lambda }&{ - \omega }\\
{\bar \omega }&\lambda
\end{array}} \right), $$
  and its norm  is defined by $$\| x \| = \sqrt {{a^2} + {b^2} + {c^2} + {d^2}}  = \sqrt {{{\left| \lambda  \right|}^2} + {{\left| \omega  \right|}^2}}. $$

\begin{remark}
Apparently, an $n\times n$ Hermitian quaternion  matrix $X = ({x_{jk}})_{n\times n}$  can be represent as a $2n \times 2n$ Hermitian matrix (see Section  2.4 in \cite{Mehta04R}).  That is,  we represent the entries of $W_n$ as
$$x_{jk}^{(n)} = \left( {\begin{array}{*{20}{c}}
a^{(n)}_{jk} + b^{(n)}_{jk}i & c^{(n)}_{jk} + d^{(n)}_{jk}i \\
-{c}^{(n)}_{jk}+{d}^{(n)}_{jk}i&{a}^{(n)}_{jk}-{b}^{(n)}_{jk}i
\end{array}} \right)
=\left( {\begin{array}{*{20}{c}}
\lambda^{(n)}_{jk} &\omega^{(n)}_{jk} \\
{ - \overline \omega }^{(n)}_{jk}&{\overline \lambda }^{(n)}_{jk}
\end{array}} \right), 1\le j<k\le n, $$
 and
$x_{jj}^{(n)} = \left( {\begin{array}{*{20}{c}}
a^{(n)}_{jj} & 0 \\
0& a^{(n)}_{jj}
\end{array}} \right).$ Then, $W_n$ is represented as a $2n\times 2n $ Hermitian complex matrix,  denoted  by $W_n^R$.  It is well known (see \cite{Zhang97Q}) that the multiplicities  of all the eigenvalues of  $W^R_n$ are even and at least 2. Taking  one from each of the $n$ pairs of eigenvalues $W_n^R$, the $n$ values are defined as the eigenvalues of $W_n$. Throughout the rest of this paper, we still use $W_n$ to denote  the
 represented one and omit the superscript $^{(n)}$ from the notations  for brevity.
 \end{remark}
The theorem can be described as following:
\begin{theorem}\label{th1}
Suppose that $W_n:=\frac{1}{{\sqrt n }}{X_n}= \frac{1}{{\sqrt n }} \({x_{jk}^{(n)}}\)_{n\times n}$, where $x_{jk}^{(n)}=a^{(n)}_{jk}+b^{(n)}_{jk} i_1+c^{(n)}_{jk} i_2+d^{(n)}_{jk} i_3$,  is a quaternion self-dual Hermitian matrix whose entries above and on the diagonal are independent and satisfy: \begin{itemize}
  \item[(i)] ${\rm E} x^{(n)}_{jk}=0,\mbox{ for all }1\le j\le k\le n.$
  \item[(ii)] ${\rm E} \|x^{(n)}_{jj}\|^2<M, ~{\rm E} \|x^{(n)}_{jk}\|^2=1, \mbox{ for all }1\le j< k\le n.$
  \item[(iii)] For any constant $\eta > 0$,
\begin{equation}\label{eq:1}
  \mathop {\lim }\limits_{{\rm{n}} \to \infty } \frac{1}{{{n^2}}}{\sum\limits_{jk} {\mbox{\rm E}\left\| {x_{jk}^{(n)}} \right\|} ^2}I(\left\| {x_{jk}^{(n)}} \right\| \ge \eta \sqrt n ) = 0.
\end{equation}
\end{itemize}
Then we have as $n\to\infty$, the ESD of $W_n$
converges to   semicircular
law almost surely.
\end{theorem}

 \begin{remark}
   Actually, the Hermitian quaternion matrix $W_n$ can be viewed as a $2n\times 2n$ Wigner complex matrix with dependent entries.  For recent
progress in this direction, we refer to \cite{GoetzeN12S,ChakrabartyH13L} and references therein.
 \end{remark}

\begin{remark}\label{rem15}
Note that condition (\ref{eq:1}) is equivalent to: for any $\eta > 0$,
   \begin{equation}\label{eq:2}
    \mathop {\lim }\limits_{n \to \infty } \frac{1}{{{\eta ^2}n^2}}{\sum\limits_{jk} {{\rm E}\| {x_{jk}^{(n)}} \|} ^2}I(\| {x_{jk}^{(n)}} \| \ge \eta \sqrt n ) = 0.
   \end{equation}
   Thus we can select a sequence ${\eta _n} \downarrow 0$ such that (\ref{eq:2}) remains true when $\eta$ is replaced by $\eta_n$.
\end{remark}

  The remainder of this paper is organized as follows. A  main mathematical tool of proving the theorem is established in Section 2.  Theorem \ref{th1} is  proved in Section 3 and some technical lemmas are given in Section 4.

\section{ The  main tool}

As in the real and complex case, we shall also use the Stieltjes transform to prove the main theorem. Note that the characteristic matrix $W^{R}_n-zI_{2n}$ of the quaternion self-dual Hermitian matrix $W^R_n$ is no longer a matrix of quaternions and neither its resolvent.  In the proof of Theorem \ref{th1}, we need the fact that the diagonal elements of the resolvent $(W_n^R-zI_{2n})^{-1}$ are pairwise identical.    
In this section,  we will give a lemma of this fact which is the key tool  in the  proof of  Theorem \ref{th1}. Before that, we introduce some definitions firstly.

\begin{definition}
A matrix is called Type-T matrix if it has the following structure:$$\left( {\begin{array}{*{20}{c}}
t&0\\
0&t
\end{array}} \right),$$
where $t$ is a complex number.\end{definition}
\begin{definition} A matrix is called Type-\uppercase\expandafter{\romannumeral1} matrix if it has the following structure:
\[\left( {\begin{array}{*{20}{c}}
{{t_1}}&0&{{a_{12}}}&{{b_{12}}}& \cdots &{{a_{1n}}}&{{b_{1n}}}\\
0&{{t_1}}&{{c_{12}}}&{{d_{12}}}& \cdots &{{c_{1n}}}&{{d_{1n}}}\\
{{d_{12}}}&{ - {b_{12}}}&{{t_2}}&0& \cdots &{{a_{2n}}}&{{b_{2n}}}\\
{ - {c_{12}}}&{{a_{12}}}&0&{{t_2}}& \cdots &{{c_{2n}}}&{{d_{2n}}}\\
 \vdots & \vdots & \vdots & \vdots & \ddots & \vdots & \vdots \\
{{d_{1n}}}&{ - {b_{1n}}}&{{d_{2n}}}&{ - {b_{2n}}}& \cdots &{{t_n}}&0\\
{ - {c_{1n}}}&{{a_{1n}}}&{ - {c_{2n}}}&{{a_{2n}}}& \ldots &0&{{t_n}}
\end{array}} \right),\]
where $t_i$, $a_{ij},b_{ij},c_{ij},d_{ij}$ are all complex numbers.
\end{definition}
\begin{definition}A matrix is called Type-\uppercase\expandafter{\romannumeral2} matrix if it has the following structure:
\begin{footnotesize}
\[\left( {\begin{array}{*{20}{c}}
{{t_1}}&0&{{a_{12}} + {c_{12}}  i}&{{b_{12}} + {d_{12}}  i}& \cdots &{{a_{1n}} + {c_{1n}}  i}&{{b_{1n}} + {d_{1n}}  i}\\
0&{{t_1}}&{ - {{\bar b}_{12}} - {{\bar d}_{12}}  i}&{{{\bar a}_{12}} + {{\bar c}_{12}}  i}& \cdots &{ - {{\bar b}_{1n}} - {{\bar d}_{1n}}  i}&{{{\bar a}_{1n}} + {{\bar c}_{1n}}  i}\\
{{{\bar a}_{12}} + {{\bar c}_{12}}  i}&{ - {b_{12}} - {d_{12}}  i}&{{t_2}}&0& \cdots &{{a_{2n}} + {c_{2n}}  i}&{{b_{2n}} + {d_{2n}}  i}\\
{{{\bar b}_{12}} + {{\bar d}_{12}}  i}&{{a_{12}} + {c_{12}}  i}&0&{{t_2}}& \cdots &{ - {{\bar b}_{2n}} - {{\bar d}_{2n}}  i}&{{{\bar a}_{2n}} + {{\bar c}_{2n}}  i}\\
 \vdots & \vdots & \vdots & \vdots & \ddots & \vdots & \vdots \\
{{{\bar a}_{1n}} + {{\bar c}_{1n}}  i}&{ - {b_{1n}} - {d_{1n}}  i}&{{{\bar a}_{2n}} + {{\bar c}_{2n}}  i}&{ - {b_{2n}} - {d_{2n}} i}& \cdots &{{t_n}}&0\\
{{{\bar b}_{1n}} + {{\bar d}_{1n}}  i}&{{a_{1n}} + {c_{1n}}  i}&{{{\bar b}_{2n}} + {{\bar d}_{2n}}  i}&{{a_{2n}} + {c_{2n}}  i}& \ldots &0&{{t_n}}
\end{array}} \right).\]
\end{footnotesize}
Here $i$ is the usual imaginary unit and all the other notations denote  complex numbers.
\end{definition}


\begin{definition} Let $A_{k1}$ and $A_{k2}$ ($1\leq k\leq n$) be  $2\times2$ complex matrices, and have the structure that   $A_{k1} = \left( {\begin{array}{*{20}{c}}
a_{k1}&a_{k2}\\
a_{k3}&a_{k4}
\end{array}} \right)_{2\times2}$
 and
$A_{k2} = \left( {\begin{array}{*{20}{c}}
a_{k4}&{ - a_{k2}}\\
{ - a_{k3}}&a_{k1}
\end{array}} \right)_{2\times2}$. Then  we denote $A_1\stackrel{d}\leftrightarrow A_2$, if
$$A_1 = \left(
{{A_{11}}},\cdots,{{A_{n1}}}
\right)_{2\times 2n},\quad  A_2 = \left(
{{A'_{12}}},\cdots,{A'_{n2}}\right)_{2n\times 2}'.$$
Let $B_{k}$ and $C_{k}$ be any $2\times2$ complex matrices, $1\leq k\leq n$. Then we denote $D_1\stackrel{u}\leftrightarrow D_2$,   if
$$D_1 = \left(
B_1+C_1\cdot i,\cdots,B_n+C_n\cdot i
\right)_{2\times 2n},$$and$$D_2 = \left(
B_1^{*'}+C_1^{*'}\cdot i,\cdots,B_n^{*'}+C^{*'}_n\cdot i
\right)_{2n\times 2}'.$$
Here superscript $'$ and $^*$ stand for the transpose and complex conjugate transpose of a matrix respectively.
\end{definition}

Here, for readers' convenient, we shall give some tips to understand these definitions. Let $M_n=(B_{jk})_{2n\times 2n}$ be a $2n\times2n$ block matrix consist of $n^2$ $2\times 2$ blocks $B_{jk}$, $j,k=1,\cdots, n$. Let $M_{(j,:)}=(B_{j,1},\cdots,B_{j,n})$ denote the $j$-th block row of $M_n$ and let $M_{(:,k)}=(B_{1,k}',\cdots,B_{n,k}')'$ denote the $k$-th block column of $M_n$, then we have
\begin{description}
  \item[(1)] if $M_n$ is a Type-\uppercase\expandafter{\romannumeral1} matrix, then for any $j$, $B_{jj}$ are all Type-T matrices and $M_{(j,:)}\stackrel{d}\leftrightarrow  M_{(:,j)}$,
  \item[(2)] if $M_n$ is a Type-\uppercase\expandafter{\romannumeral2} matrix, then for any $j$, $B_{jj}$ are all Type-T matrices and $M_{(j,:)}\stackrel{u}\leftrightarrow  M_{(:,j)}$.
\end{description}

Now we are in position to present the following important lemma which is the main tool to prove Theorem \ref{th1}.

\begin{lemma}\label{lemma:1}
For all $n\geq1$, if $\Omega_n$ is  an invertible complex $2n\times 2n$ matrix of Type-\uppercase\expandafter{\romannumeral2} matrix, then $\Omega_n^{-1}$ is of Type-\uppercase\expandafter{\romannumeral1}.
\end{lemma}
\begin{proof}

We will prove the lemma  by induction.
  First of all,  we can easily verify that the conclusion is correct when $n =1$ and $2$.
  Now,  suppose the conclusion is true when $n = m\left( m\ge 2\right)$. Then let $n = m + 1$ and suppose $t_1 \neq 0$.  Otherwise, one may add a small constant $\varepsilon>0$ to $t_1$ and then let $\varepsilon \to 0$ in the resulting Type I matrix. 

  Write:
$\begin{array}{l}
\Omega_{m+1} = \left( {\begin{array}{*{20}{c}}
{{\Sigma _{11}}}&{{\Sigma _{12}}}\\
{{\Sigma _{21}}}&{{\Sigma _{22}}}
\end{array}} \right)
\end{array}$ , where ${\Sigma _{11}} = \left( {\begin{array}{*{20}{c}}
{{t_1}}&0\\
0&{{t_1}}
\end{array}} \right)$,
$$\Sigma_{12}=\left(\begin{array}{ccccc}
a_{12}+c_{12}i&b_{12}+d_{12}i&\cdots&a_{1,m+1}+c_{1,m+1}i&b_{1,m+1}+d_{1,m+1}i\\
-\bar b_{12}-\bar d_{12}i&\bar a_{12}+\bar c_{12}i&\cdots&-\bar b_{1,m+1}-\bar d_{1,m+1}i&\bar a_{1,m+1}+\bar c_{1,m+1}i
\end{array}\right),$$
$$\Sigma_{21}=\left(\begin{array}{ccccc}
\bar a_{12}+\bar c_{12}i&\bar b_{12}+\bar d_{12}i&\cdots&\bar a_{1,m+1}+\bar c_{1,m+1}i&\bar b_{1,m+1}+\bar d_{1,m+1}i\\
-b_{12}- d_{12}i&a_{12}+c_{12}i&\cdots&- b_{1,k+1}- d_{1,m+1}i&a_{1,m+1}+ c_{1,m+1}i
\end{array}\right)',$$
$${\Sigma _{22}}
= \left( {\begin{array}{*{20}{c}}
{{t_2}}&0&{{a_{23}} + {c_{23}}i}&{{b_{23}} + {d_{23}}i}& \cdots\\
0&{{t_2}}&{ - {{\bar b}_{23}} - {{\bar d}_{23}}  i}&{{{\bar a}_{23}} + {{\bar c}_{23}}  i}& \cdots \\
{{{\bar a}_{23}} + {{\bar c}_{23}} i}&{ - {b_{23}} - {d_{23}}  i}&{{t_3}}&0& \cdots \\
{{{\bar b}_{23}} + {{\bar d}_{23}} i}&{{a_{23}} + {c_{23}}  i}&0&{{t_3}}& \cdots \\
 \vdots & \vdots & \vdots & \vdots & \ddots \\
\end{array}} \right)_{2m \times 2m} .
$$

According to Lemma \ref{lemma:2},  to complete the proof  it is sufficient to show that:
\begin{description}
  \item[(1)] ${\Sigma _{22}} - {\Sigma _{21}}\Sigma _{11}^{ - 1}{\Sigma _{12}}$ is a Type-\uppercase\expandafter{\romannumeral2} matrix.
  \item[(2)] ${\Sigma _{11}^{ - 1} + \Sigma _{11}^{ - 1}\Sigma _{12}^{}\Sigma _{22. 1}^{ - 1}\Sigma _{21}^{}\Sigma _{11}^{ - 1}}$ is a Type-T matrix.
  \item[(3)] ${ - \Sigma _{11}^{ - 1}\Sigma _{12}^{}\Sigma _{22. 1}^{ - 1}} \stackrel{d}\leftrightarrow { - \Sigma _{22. 1}^{ - 1}\Sigma _{21}^{}\Sigma _{11}^{ - 1}}$ .
\end{description}
We now proceed in our proof by taking  these three steps. Let
$A_{jk}=\left(\begin{array}{cc}
a_{jk} & b_{jk}\\
-\overline{b}_{jk} & \overline{a}_{jk}
\end{array}\right)_{2 \times 2},$
$B_{jk}=\left(\begin{array}{cc}
c_{jk} & d_{jk}\\
-\overline{d}_{jk} & \overline{c}_{jk}
\end{array}\right)_{2 \times 2}$ for  $1\le j<k\leq m+1$.

\emph{\textbf{Step 1: Proof of (1)}}. Apparently,  $\Sigma _{22}$ is a Type-\uppercase\expandafter{\romannumeral2} matrix and  $\Sigma _{11}^{ - 1}$ is a scalar complex matrix.  What's more,  we can easily verify that if two $2\times2$ matrices $A \stackrel{u}\leftrightarrow B$ , then for any $t \in \mathbb{C}$, $t  A \stackrel{u}\leftrightarrow t  B$.
     Therefore,  it's sufficient to show that ${\Sigma _{21}}{\Sigma _{12}}$ is a Type-\uppercase\expandafter{\romannumeral2} matrix.  Rewrite
$${\Sigma _{21}}{\Sigma _{12}}
= \left( {\begin{array}{*{20}{c}}
{A_{12}^* + B_{12}^* i}\\
 \vdots \\
{A_{1,m+1}^* + B_{1,m+1}^*  i}
\end{array}} \right)\left( {\begin{array}{*{20}{c}}
{{A_{12}} + {B_{12}}  i},  & \cdots , &{{A_{1,m+1}} + {B_{1,m+1}} i}
\end{array}} \right).$$
 Write $\Gamma={\Sigma _{21}}{\Sigma _{12}}={({\gamma _{jk}})_{j, k = 2,  \cdots , m+1}}$ , where ${\gamma _{jk}}$ is a $2\times2$ complex matrix.

   For $j=2,\cdots, m+1,$ the diagonal $2\times 2$ block entries $\gamma _{jj}$ are
         \begin{eqnarray}
         &({A_{1j}^*} + {B_{1j}^*}  i) \cdot (A_{1j} + B_{1j}  i) \notag
         = ({A_{1j}^*}A_{1j} - {B_{1j}^*}B_{1j}) + ({A_{1j}^*}B_{1j} + {B_{1j}^*}A_{1j})  i \notag \\
         &=\left(|a_{1j}|^2+|b_{1j}|^2-|c_{1j}|^2-|d_{1j}|^2+2i\Re(\bar a_{1j}c_{1j}+b_{1j}\bar d_{1j})\right)I_2.
          \label{eqproduct}
          \end{eqnarray}
 Thus $\gamma _{jj}$ are all Type-T matrices. \\
 {\bf Note:} {\it In the equation (\ref{eqproduct}), if we consider the $2\times 2$ matrices $A_{1j}$ and $B_{1j}$ as quaternions and their algebraic operations as those in quaternions, then on the right hand side of the first line of (\ref{eqproduct}), the expressions  $\|A_{1j}\|^2-\|B_{1j}\|^2$ and $A_{1j}^*B_{1j}+B_{1j}^*A_{1j}
 =2\Re(A_{1j}^*B_{1j})$ are also quaternions. For simplicity of expressions, we shall use either of the double interpretations of the matrices $A_{jk}$ and $B_{jk}$ inaccordance with convenience in the following arguments.}

        Next, for $j\neq k$ we have
        \begin{align}
        ({A_{1j}^*} + {B_{1j}^*} i) \cdot (A_{1k} + B_{1k}  i) = ({A_{1j}^*}A_{1k} - {B_{1j}^*}B_{1k}) + ({A_{1j}^*}B_{1k} + {B_{1j}^*}A_{1k})  i\notag
        \end{align}and
        \begin{align}
        ({A_{1k}^*} + {B_{1k}^*}  i) \cdot (A_{1j} + B_{1j} i)
        &= ({A_{1k}^*}A_{1j} - {B_{1k}^*}B_{1j}) + ({A_{1k}^*}B_{1j} + {B_{1k}^*}A_{1j})  i \notag \\
        &= {({A_{1j}^*}A_{1k} - {B_{1j}^*}B_{1k})^*} + {({A_{1j}^*}B_{1k} + {B_{1j}^*}A_{1k})^*}  i \notag,
        \end{align} which implies,  for any $2\leq j< k\leq m+1$ , ${\gamma _{jk}} \stackrel{u}\leftrightarrow {\gamma _{kj}}$.
      Thus,  the proof of {\textbf (1)} is  complete.

\emph{\textbf{Step 2: Proof of (2)}}. Since ${\Sigma _{11}^{ - 1}}$ is a Type-T matrix,  we only need to prove that ${\Sigma _{12}}\Sigma _{22. 1}^{ - 1}{\Sigma _{21}}$ is also a Type-T matrix.  Write \[\Sigma _{22. 1}^{ - 1} = \left( {\begin{array}{*{20}{c}}
{{P_{11}}}&{{P_{12}}}& \cdots &{{P_{1, m}}}\\
{{P_{21}}}&{{P_{22}}}& \cdots &{{P_{2, m}}}\\
 \vdots & \vdots & \ddots & \vdots \\
{{P_{m, 1}}}&{{P_{m, 2}}}& \cdots &{{P_{m, m}}}
\end{array}} \right) \]
where $P_{jk}$ are $2\times 2$ matrices. Then we have, $${\Sigma _{12}}\Sigma _{22. 1}^{ - 1}{\Sigma _{21}} = \sum\limits_{jk} ({{A_{1,j+1}+B_{1,j+1}i}){P_{jk}}({A_{1,k+1}^*}+B_{1,k+1}^*i}).$$
           From the induction hypothesis,  we know ${{P_{jj}}}$ is a Type-T matrix and ${{P_{jk}}}\stackrel{d}\leftrightarrow{{P_{kj}}}$.

 If $j = k$,  since ${P _{jj}}$ is a Type-T matrix and by the note above 
             \begin{align}
         &({A_{1,j+1}} + {B_{1,j+1}}  i) \cdot (A_{1,j+1}^* + B_{1,j+1}^* i) \notag \\
         &= ({A_{1,j+1}}A_{1,j+1}^* - {B_{1,j+1}}B_{1,j+1}^*) + ({A_{1,j+1}}B_{1,j+1}^* + {B_{1,j+1}}A_{1,j+1}^*)  i \notag \\
         &=\big((\|A_{1,j+1}\|^2-\|B_{1,j+1}\|^2)+2\Re({A_{1,j+1}}B_{1,j+1}^*)i\big)I_2\notag,
          \end{align}  we get that $({A_{1,j+1}+B_{1,j+1}i}){P_{jj}}({A_{1,j+1}^*}+B_{1,j+1}^*i)$  is a Type-T matrix.

If $j \neq k$, then we have
\begin{align}
&({A_{1,j+1}} + {B_{1,j+1}}  i){P_{jk}}(A_{1,k+1}^* + B_{1,k+1}^*  i) + ({A_{1,k+1}} + {B_{1,k+1}}  i){P_{kj}}(A_{1,j+1}^* + B_{1,j+1}^*  i)\notag \\
 = &({A_{1,j+1}}{P_{jk}}A_{1,k+1}^* + {A_{1,k+1}}{P_{kj}}A_{1,j+1}^*) - ({B_{1,j+1}}{P_{jk}}B_{1,k+1}^* + {B_{1,k+1}}{P_{kj}}B_{1,j+1}^*) \notag \\&+ ({A_{1,j+1}}{P_{jk}}B_{1,k+1}^* + {B_{1,k+1}}{P_{kj}}A_{1,j+1}^*)  i + ({B_{1,j+1}}{P_{jk}}A_{1,k+1}^* + {A_{1,k+1}}{P_{kj}}B_{1,j+1}^*)  i\notag.
\end{align}
Since ${{P_{jk}}}\stackrel{d}\leftrightarrow{{P_{kj}}}$, thus we can assume
\begin{displaymath}
P_{jk}=
\left(\begin{array}{cc}
e_{jk} & g_{jk}\\
h_{jk} & f_{jk}
\end{array}\right),\quad
P_{kj}=
\left(\begin{array}{cc}
f_{jk} & -g_{jk}\\
-h_{jk} & e_{jk}
\end{array}\right).
\end{displaymath}
It is easy to verify that
\begin{align}
&({A_{1,j+1}}{P_{jk}}A_{1,k+1}^* + {A_{1,k+1}}{P_{kj}}A_{1,j+1}^*)\notag \\
=&\left(\begin{array}{cc}
a_{1,j+1} & b_{1,j+1}\\
-\overline{b}_{1,j+1} & \overline{a}_{1,j+1}
\end{array}\right)
\left(\begin{array}{cc}
e_{jk} & g_{jk}\\
h_{jk} & f_{jk}
\end{array}\right)
\left(\begin{array}{cc}
\overline{a}_{1,k+1} & -b_{1,k+1}\\
\overline{b}_{1,k+1} & a_{1,k+1}
\end{array}\right)\notag \\
+&\left(\begin{array}{cc}
a_{1,k+1} & b_{1,k+1}\\
-\overline{b}_{1,k+1} & \overline{a}_{1,k+1}
\end{array}\right)
\left(\begin{array}{cc}
f_{jk} & -g_{jk}\\
-h_{jk} & e_{jk}
\end{array}\right)
\left(\begin{array}{cc}
\overline{a}_{1,j+1} & -b_{1,j+1}\\
\overline{b}_{1,j+1} & a_{1,j+1}
\end{array}\right)\notag \\
=&\left(\begin{array}{cc}
2x & 0\\
0 & 2x
\end{array}\right),\notag
\end{align}
where \begin{align*}
x=(\overline{a}_{1,j+1}a_{1,k+1}+b_{1,j+1}\overline{b}_{1,k+1})f_{jk}+(a_{1,j+1}\overline{a}_{1,k+1}+\overline{b}_{1,j+1}b_{1,k+1})e_{jk}\\
+(a_{1,j+1}\overline{b}_{1,k+1}-\overline{b}_{1,j+1}a_{1,k+1})g_{jk}+(b_{1,j+1}\overline{a}_{1,k+1}-\overline{a}_{1,j+1}b_{1,k+1})h_{jk}.
\end{align*}
Hence, we conclude that $({A_{1,j+1}}{P_{jk}}A_{1,k+1}^* + {A_{1,k+1}}{P_{kj}}A_{1,j+1}^*)$ is a Type-T matrix. Similarly, we can verify that $({B_{1,j+1}}{P_{jk}}B_{1,k+1}^* + {B_{1,k+1}}{P_{kj}}B_{1,j+1}^*)$, $({A_{1,j+1}}{P_{jk}}B_{1,k+1}^*+{B_{1,k+1}}{P_{kj}}A_{1,j+1}^*)$,  and $({B_{1,j+1}}{P_{jk}}A_{1,k+1}^* + {A_{1,k+1}}{P_{kj}}B_{1,j+1}^*)$ are all Type-T matrices, which complete the proof.

\emph{\textbf{Step 3: Proof of (3)}}. Since  $\Sigma _{11}^{ - 1}$ is a diagonal matrix,  thus we only need to prove that ${\Sigma _{12}}\Sigma _{22. 1}^{ - 1}\stackrel{d}\leftrightarrow\Sigma _{22. 1}^{ - 1}{\Sigma _{21}}$ . Write ${Q} = {\Sigma _{12}}\Sigma _{22. 1}^{ - 1} = \begin{pmatrix}Q_1&{ Q_2}&\cdots  &Q_{m}
\cr \end{pmatrix}$,
$V = \Sigma _{22. 1}^{ - 1}{\Sigma _{21}} = \begin{pmatrix}
{{V'_1}}&{ {V'_2}}&{  \cdots  }&{{V'_{m}}}
\cr \end{pmatrix}'$ and
\[\Sigma _{22. 1}^{ - 1} = \left( {\begin{array}{*{20}{c}}
{{P_{11}}}&{{P_{12}}}& \cdots &{{P_{1,m}}}\\
{{P_{21}}}&{{P_{22}}}& \cdots &{{P_{2, m}}}\\
 \vdots & \vdots & \ddots & \vdots \\
{{P_{m, 1}}}&{{P_{m, 2}}}& \cdots &{{P_{m, m}}}
\end{array}} \right), \]
Then for any $k$, we  have  \begin{align} \label{align:1}
Q_k = \sum_{j}^{} (A_{1,j+1}+B_{1,j+1}i) P_{jk}, V_k = \sum_{j}^{} {{P_{kj}}(A_{1,j+1}^*+B_{1,j+1}^*i)}.
\end{align}
To complete the proof,  it is sufficient to show that for any $k$,  $Q_k\stackrel{u}\leftrightarrow V_k$ .
From the induction hypothesis,  we  assume that ${P_{jk}} = \left( {\begin{array}{*{20}{c}}
{{e_{jk}}}&{{g_{jk}}}\\
{{h_{jk}}}&{{f_{jk}}}
\end{array}} \right)$,
${P_{kj}} = \left( {\begin{array}{*{20}{c}}
{{f_{jk}}}&{ - {g_{jk}}}\\
{ - {h_{jk}}}&{{e_{jk}}}
\end{array}} \right)$.\\
Then we have
\begin{align}\label{align:2}
&(A_{1,j+1} + B_{1,j+1}  i) \cdot \left( {\begin{array}{*{20}{c}}
{{e_{jk}}}&{{g_{jk}}}\\
{{h_{jk}}}&{{f_{jk}}}
\end{array}} \right)\notag \\
= &A_{1,j+1}  \left( {\begin{array}{*{20}{c}}
{{e_{jk}}}&{{g_{jk}}}\\
{{h_{jk}}}&{{f_{jk}}}
\end{array}} \right) + B_{1,j+1}  \left( {\begin{array}{*{20}{c}}
{{e_{jk}}}&{{g_{jk}}}\\
{{h_{jk}}}&{{f_{jk}}}
\end{array}} \right)  i\notag \\
 = &\left( {\begin{array}{*{20}{c}}
{{e_{jk}}{a_{1,j+1}} + {h_{jk}}{b_{1,j+1}}}&{{g_{jk}}{a_{1,j+1}} + {f_{jk}}{b_{1,j+1}}}\\
{ - {e_{jk}}{{\overline b}_{1,j+1}} + {h_{jk}}{{\overline a}_{1,j+1}}}&{ - {g_{jk}}{{\overline b}_{1,j+1}} + {f_{jk}}{{\overline a}_{1,j+1}}}
\end{array}} \right)\notag\\ +& \left( {\begin{array}{*{20}{c}}
{{e_{jk}}{c_{1,j+1}} + {h_{jk}}{d_{1,j+1}}}&{{g_{jk}}{c_{1,j+1}} + {f_{jk}}{d_{1,j+1}}}\\
{ - {e_{jk}}{{\overline d}_{1,j+1}} + {h_{jk}}{{\overline b}_{1,j+1}}}&{ - {g_{jk}}{{\overline d}_{1,j+1}} + {f_{jk}}{{\overline c}_{1,j+1}}}
\end{array}} \right)  i\notag
\end{align}
\\and
\begin{align}
 &\left( {\begin{array}{*{20}{c}}
{{f_{jk}}}&{ - {g_{jk}}}\\
{ - {h_{jk}}}&{{e_{jk}}}
\end{array}} \right)  ({A_{1,j+1}^*} + {B_{1,j+1}^*}  i)\notag \\
 = &\left( {\begin{array}{*{20}{c}}
{{f_{jk}}}&{ - {g_{jk}}}\\
{ - {h_{jk}}}&{{e_{jk}}}
\end{array}} \right)  {A_{1,j+1}^*} + \left( {\begin{array}{*{20}{c}}
{{f_{jk}}}&{ - {g_{jk}}}\\
{ - {h_{jk}}}&{{e_{jk}}}
\end{array}} \right)  {B_{1,j+1}^*}  i\notag \\
 = &\left( {\begin{array}{*{20}{c}}
{ - {g_{jk}}{{\overline b}_{1,j+1}} + {f_{jk}}{{\overline a}_{1,j+1}}}&{ - {g_{jk}}{a_{1,j+1}} - {f_{jk}}{b_{1,j+1}}}\\
{{e_{jk}}{{\overline b}_{1,j+1}} - {h_{jk}}{{\overline a}_{1,j+1}}}&{{e_{jk}}{a_{1,j+1}} + {h_{jk}}{b_{1,j+1}}}
\end{array}} \right) \notag \\+& \left( {\begin{array}{*{20}{c}}
{ - {g_{jk}}{{\overline d}_{1,j+1}} + {f_{jk}}{{\overline c}_{1,j+1}}}&{ - {g_{jk}}{c_{1,j+1}} - {f_{jk}}{d_{1,j+1}}}\\
{{e_{jk}}{{\overline d}_{1,j+1}} - {h_{jk}}{{\overline b}_{1,j+1}}}&{{e_{jk}}{c_{1,j+1}} + {h_{jk}}{d_{1,j+1}}}
\end{array}} \right)  i,\notag
\end{align}
which together with (\ref{align:1}) complete the proof of (3). Therefore, we get Lemma \ref{lemma:1} with $t_1\neq0$.

Since all the entries of $\Omega_m^{-1}$ are in fact a polynomial of $t_1$, thus by the continuity of polynomial, we have the conclusion of
Lemma \ref{lemma:1} is  true even $t_1=0$. Then we complete the proof.
\end{proof}
\section{Proof of Theorem \ref{th1}}

In this section we give the proof of Theorem \ref{th1}. The tools we use here are  Stieltjes transform and  Burkholder inequality for the martingale difference sequence. The proof is following the same steps as Section 2 in \cite{BaiS10S}.

\begin{remark}
In Wigner's paper, the semicircular law of random matrices with real entries was proved by the moment method. And one can find that Wigner's approach can be applied to the complex Hermitian ensembles. But when the entries of random matrices are quaternions, the moment method will run into a stone wall due to the breaking of the commutative law of multiplication in the quaternion field. What's more, the moment method cannot give any convergence rate. Thus in this paper we will use the method of Stieltjes transform to prove our main theorem. And we believe that our Lemma \ref{lemma:1} will play an important role in further studies of the common quaternion self-dual Hermitian matrices.
\end{remark}

\subsection{ Truncation, centralization and rescale}

   Define $${\widetilde W}_n =n^{-1/2}(\tilde x_{ij})_{j, k=1}^n =\frac{1}{{\sqrt n }}(x_{jk}I(\| {x_{jk}} \| \le {\eta _n}\sqrt n ))_{j, k=1}^n.$$ Then
   by Lemma \ref{lemma:3}, we obtain that
\begin{equation}\label{eq:3}
  \left\| {{F^{{W_n}}} - {F^{\widetilde W_n}}} \right\| \le \frac{1}{{2n}}{\rm {rank}}\left( {{W_n} - {{\widetilde W}_n}} \right) \leq \frac{1}{{n}}\sum\limits_{1\leq j \leq k \leq n } {I(\| {x_{jk}} \| \ge {\eta _n}\sqrt n )}.
\end{equation}
 By condition (\ref{eq:2}),  we have $${\rm E}\left(\frac{1}{n}\sum\limits_{1 \le j \le k \le n} {I(\| {x_{jk}} \| \ge {\eta _n}\sqrt n )}\right ) \le \frac{1}{{\eta _n^2{n^2}}}{\sum\limits_{jk} {{\rm E}\| {x_{jk}} \|} ^2}I(\| {x_{jk}} \| \ge {\eta _n}\sqrt n ) = o(1), $$
 and $${\rm {Var}}\left(\frac{1}{n}\sum\limits_{1 \le j \le k \le n} {I(\| {x_{jk}} \| \ge {\eta _n}\sqrt n )}\right ) \le \frac{1}{{\eta _n^2{n^3}}}{\sum\limits_{jk} {{\rm E}\| {x_{jk}} \|} ^2}I(\| {x_{jk}} \| \ge {\eta _n}\sqrt n ) = o(\frac{1}{n}). $$
 Then by Bernstein's inequality,  for all small $\varepsilon > 0 $ and large $n$,  we have
 \begin{equation}\label{eq:4}
   {\rm P}\left(\frac{1}{n}\sum\limits_{1 \le j \le k \le n} {I(\| {x_{jk}} \| \ge {\eta _n}\sqrt n )}  \ge \varepsilon \right) \le 2{e^{ - \varepsilon n}},
 \end{equation}
 which is summable.  Thus combining (\ref{eq:3}), (\ref{eq:4}) and Borel-Cantelli Lemma, we obtain
  \begin{equation}
  \left\| {{F^{{W_n}}} - {F^{\widetilde W_n}}} \right\| \to 0, \quad a. s.
\end{equation}

Next we will
remove the diagonal elements.  Let $\widehat W_n$ be the matrix obtained from $\widetilde W_n$ by replacing the diagonal elements with 0.  Then using Lemma \ref{lemma:4},  we have:
\begin{gather*}
{L^3}\left( {{F^{{{\widehat W}_n}}}, {F^{{{\widetilde W}_n}}}} \right) \le \frac{1}{{2n}}{\rm{tr}} {{{\left( {{{\widetilde W}_n} - {{\widehat W}_n}} \right)}}}{{{\left( {{{\widetilde W}_n} - {{\widehat W}_n}} \right)}^*}} \\ \le \frac{1}{{{n^2}}}\sum\limits_{j = 1}^n \| {{x_{jj}}} \|^2 I\left( {\| {{x_{jj}}} \| < {\eta _n}\sqrt n} \right) \le \eta _n^2 \to 0,
\end{gather*}
where $L(\cdot, \cdot)$ is the Levy distance between two distributions (See Remark A.39. in \cite{BaiS10S}).

In addition, by Lemma \ref{lemma:4} and Remark \ref{rem15},  we have:
\begin{align*}
&{L^3}({F^{{\widehat W_n}}},  {F^{{\widehat W_n} - {\rm E}{\widehat W_n}}})
 \le \frac{1}{{{n^2}}}{\sum\limits_{j \ne k} {\left\| {{\rm E}({x_{jk}}I(\| {{x_{jk}}}\| \le {{\eta _n}\sqrt n })}) \right\|} ^2}\\
 \le& \frac{1}{{{\eta_n^2n^2}}}\sum\limits_{jk} {{\rm E}{{\| {{x_{jk}}} \|}^2}I(\| {{x_{jk}}} \| \ge {{\eta _n}\sqrt n })}  \to 0.
\end{align*}

Now the remaining work
  is rescaling. Let  $\sigma _{jk}^2 = {\rm{Var(}}\tilde x_{jk} )$. If $j\le k$ and $\si_{jk}^2<1/2$, then we replace $\tilde x_{jk} - {\rm E} \tilde x_{jk}$ by a bounded real random variable ${\breve{x}_{jk}}$ with mean 0, variance 1, say the one taking $\pm 1$ with probability $1/2$, and  being  independent of the other entries. If $j\le k$ and $\si_{jk}^2\geq1/2$, we  denote   ${\breve{x}_{jk}}=\tilde x_{jk} - {\rm E} \tilde x_{jk}$. And if $j\geq k$, we  denote   ${\breve{x}_{jk}}={\bar{\breve{x}}_{kj}}$. Let $E_n$ be the set of pairs $\{(j, k)$: $\sigma _{jk}^2 < \frac{1}{2}\}$ and $N_n$ be the cardinal number of $E_n$.  Because $\frac{1}{n^2}\sum \limits_{jk} \sigma _{jk}^2 \to 1 $, combining with the fact that $\sigma_{ij}^2 \leq 1$ for any $j,k$, we can conclude that $N_n=o(n^2)$.  Write ${\breve{q}_{jk}}={\breve{x}_{jk}}-\tilde x_{jk} + {\rm E} \tilde x_{jk}$ and write $\breve{W}_n= \frac{1}{\sqrt n}({\breve{x}_{jk}})$.  By Lemma \ref{lemma:4},  we get that
  \begin{equation}\label{eq:9}
{L^3}(F^{\breve W_n} ,  F^{{\widehat W_n} - {\rm E}{\widehat W_n}})
 \le {\frac{1}{{{n^2}}}\sum\limits_{(j,k)\in E_n} \|\breve{q}_{jk}\|^2}.
\end{equation}

Rewrite the right hand side above as $${\frac{1}{{{n^2}}}\sum\limits_{(j,k)\in E_n} \|\breve{q}_{jk}\|^2}:=\frac{2}{n^2}\sum\limits_{k=1}^K u_k,$$  where $K=\frac {1}{2}N_n$.  Then,  select $m=[\log n]$ and for any fixed $t, \varepsilon >0$,  we have:
\begin{equation}
\begin{split}
{\rm P}\left(\frac2{n^2}\sum_{k=1}^Ku_k\ge \varepsilon\right)&\le {\rm E}\left(\frac{2}{\varepsilon n^2}\sum\limits_{k=1}^K u_k\right)^m \\
&=\frac{2^m}{\varepsilon^mn^{2m}}\sum\limits_{m_1+\ldots+m_k=m}\frac{m!}{{m_1}!\ldots{m_k}!}{\rm E}u_1^{m_1}\ldots {\rm E}u_K^{m_k} \\
&\leq \frac{2^m}{\varepsilon^mn^{2m}}\sum\limits_{l=1}^m \sum_{\substack{m_1+\ldots +m_l=m \\ m_t\geq 1}}\frac{m!}{l!m_1! \ldots m_l!}\prod_{t=1}^l \left(\sum_{k=1}^K {\rm E}u_k^{m_t}\right) \\
&\leq C\varepsilon^{-m} \sum_{l=1}^m 2^m n^{-2m}l^m(l!)^{-1}(2\eta_n^2 n)^{m-l}2^l K^l \\
&\leq C\varepsilon^{-m}\sum_{l=1}^m\left(\frac{12K}{n^2}\right)^l\left(\frac{4\eta_n^2 l}{n}\right)^{m-l} \\
&\leq C\left(\frac{12K}{n^2\varepsilon}+\frac{4\eta_n^2 m}{n\varepsilon}\right)^m = o(n^{-t}),\notag
\end{split}
\end{equation}
where we have used the fact that for all $l\geq 1$,  $l!\geq(l/3)^l$ and the last inequality follows from facts that the two terms in the parentheses tend to 0 and $m=[ \log n]$.  From the inequality above with $t=2$ and (\ref{eq:9}) we conclude that :
$$L(F^{\breve W_n} ,  F^{{\widehat W_n} - {\rm E}{\widehat W_n}}) \to 0,  a. s.. $$
Write$$\widetilde{\widetilde{W}}_n=\frac{1}{\sqrt n} \Big(\breve \sigma_{jk}^{-1}\breve x_{jk}(1-\delta_{jk})\Big),$$
where $\breve \sigma_{jk}^2={\rm E}\|\breve x_{jk}\|^2$ and $\delta_{jk}$ is the Kronecker delta,  i.e.  equal to 1 when $j=k$ and 0 otherwise.
By Lemma \ref{lemma:4},  it follows that:
 \begin{equation}
{L^3}(F^{\widetilde{\widetilde{W}}_n} ,  F^{\breve W_n})
\le {\frac{1}{{{n^2}}}\sum_{i \neq j} (1-\breve \sigma_{ij}^{-1})^2\|{\breve x}_{ij}\|^2}
+\frac1{n^2}\sum_{j=1}^n\|\breve x_{jj}\|^2.\nonumber
\end{equation}
Note that
 \begin{equation}
 \begin{split}
&{\rm E}\Big ( {\frac{1}{{{n^2}}}\sum_{j \neq k} (1-\breve \sigma_{jk}^{-1})^2\|{\breve x}_{jk}\|^2}+\frac1{n^2}\sum_{j=1}^n\|\breve x_{jj}\|^2 \Big )\\
=& \frac{1}{{{n^2}}}\sum_{j\ne k} (1-\breve \sigma_{jk})^2+\frac{M}{n}
\leq \frac{1}{{{n^2}}}\sum_{j\ne k} (1-\breve \sigma_{jk}^2) +\frac{M}{n}\\
\leq & \frac{1}{{{n^2}}}\sum_{(j, k) \not \in E_n } \left[{\rm E}\|x_{jk}\|^2 I(\|x_{jk}\|\geq \eta_n \sqrt n)+{\rm E}^2\|x_{jk}\|I(\|x_{jk}\|\geq \eta_n \sqrt n)\right]+\frac{M}{n}\\
 \to& 0.\notag
\end{split}
\end{equation}
Also,  applying Lemma \ref{lemma:5},  we have:
 \begin{equation}
 \begin{split}
&{\rm E }  \Bigg | \frac{1}{n^2}\sum_{j\neq k} (1-\breve \sigma_{jk}^{-1})^2 \Big ( \|{\breve x}_{jk}\|^2 - {\rm E}({\|\breve x}_{jk}\|^2) \Big ) \Bigg |^4 \\
\leq & \frac{C}{n^8}\Bigg [\sum_{j \neq k}{\rm E}\|x_{jk}\|^8 I(\|x_{jk}\|\leq \eta_n \sqrt n)+\Big ( \sum_{j\neq k} {\rm E}\|x_{jk}\|^4 I(\|x_{jk}\|\leq \eta_n \sqrt n)\Big )^2 \Bigg ] \\
\leq & Cn^{-2}[n^{-1}\eta_n^6+\eta_n^4] \notag
\end{split}
\end{equation}
which is summable.  Similarly, we have 
 \begin{equation}
 \begin{split}
&{\rm E }  \Bigg | \frac{1}{n^2}\sum_{j=1}^n  \Big ( \|{\breve x}_{jj}\|^2 - {\rm E}({\|\breve x}_{jj}\|^2) \Big ) \Bigg |^2 \\
\leq & \frac{C}{n^4}\Bigg [\sum_{j =1}^n{\rm E}\|x_{jk}\|^4 I(\|x_{jj}\|\leq \eta_n \sqrt n) \Bigg ] \\
\leq & Cn^{-2}\eta_n^2 \notag
\end{split}
\end{equation}
which is also summable.
From the three estimates above,  we conclude that
$$L(F^{\widetilde{\widetilde{W}}_n} ,  F^{\breve W_n}) \to 0,  a. s.. $$
Therefore,  we conclude that:
$${L}(F^{{ W_n}} , F^{\widetilde{\widetilde{W}}_n})\rightarrow 0,  \quad a.s. . $$

Thus in the proof of the theorem, we may assume that:
\begin{enumerate}
                     \item The variables ${x_{jk}, 1\le j<k \le n}$ are independent and ${x_{jj}} = \left(
                                        \begin{array}{cc}
                                          0 & 0 \\
                                          0 & 0 \\
                                        \end{array}
                                      \right).\\$
                     \item ${\rm E}x_{jk}=0, {\rm Var}(x_{jk})=1, 1 \le j < k \le n.$
                     \item $\left\| {{x_{jk}}} \right\|\leq \eta_n \sqrt n, 1 \le j < k \le n$.
                   \end{enumerate}
For brevity, we still use $x_{ij}$ to denote the truncated and normalized variables  in the sequel.

\subsection{Proof of Theorem \ref{th1}}
The main mathematical tool of the proof of Theorem \ref{th1} is
 Stieltjes transform, which is defined as:
For any function of bounded variation $G$ on the real line, its Stieltjes transform is defined by
$$s_G(z)=\int\frac{1}{y-z}dG(y),~~z\in\mathbb{C}^{+}\equiv\{z\in\mathbb{C}:\Im z>0\}.$$
From Theorems B.8-B.10 in \cite{BaiS10S},   we conclude that we just need to  proceed in our proof by the following three steps:
\begin{description}
  \item[1 ] For any fixed $z \in {\mathbb{C}^ + }$,  ${s_n}(z) - {\rm E}{s_n}(z) \to 0,  \quad a. s..  $
  \item[2 ] For any fixed $z \in {\mathbb{C}^ + }$,  ${\rm E}{s_n}(z) \to s(z) $.
  \item[3 ] Outside a null set,  ${s_n}(z) \to s(z)$ for every $z \in {\mathbb{C}^ + }$.
\end{description}
Here $z=u+\upsilon i$ with $\upsilon>0$, $s_n(z)=s_{F^{W_n}}(z)$ and $s(z)$ is
  the Stieltjes transform of the semicircular law (see Lemma \ref{lemma:6}). Similar to the Step 3 in  Section 2.3 of \cite{BaiS10S}, the last step is implied by the first two steps and thus
its proof is omitted. We now proceed with the first two steps.

{\bf{ Step 1}}:
Let ${\rm E_k}( \cdot )$ denote the conditional expectation given $\left\{ {{x_{jl}}, j, l > k} \right\}$,  then we have
\label{eq:7}
  \begin{align}
{s_n}(z) - {\rm E}{s_n}(z) &= \frac{1}{{2n}}\sum\limits_{k = 1}^n[{{{\rm E}_{k - 1}}{\rm tr}{{({W_n} - z{I_{2n}})}^{ - 1}} - } {\rm E_k}{\rm tr}{({W_n} - z{I_{2n}})^{ - 1}}]\no\\&= \frac{1}{{2n}}\sum\limits_{k = 1}^n {{\gamma _k}} ,
 \end{align}
where
  \begin{equation}\label{eq:8}
  \begin{split}
{\gamma _k} =& {{\rm E}_{k - 1}}{\rm tr}{({W_n} - z{I_{2n}})^{ - 1}} - {{\rm E}_k}{\rm tr}{({W_n} - z{I_{2n}})^{ - 1}}\\
 =& {{\rm E}_{k - 1}}[{\rm tr}{({W_n} - z{I_{2n}})^{ - 1}} - {\rm tr}{(W_n^{(k)} - z{I_{2n - 2}})^{ - 1}}]\\
 & - {{\rm E}_k}[{\rm tr}{({W_n} - z{I_{2n}})^{ - 1}} - {\rm tr}{(W_n^{(k)} - z{I_{2n - 2}})^{ - 1}}]
 \end{split}
\end{equation}
and $W_n^{(k)}$ is the matrix obtained from ${W_n}$ with the $k$-th quaternion row and quaternion column removed. Notice that here we use the fact that
\begin{align*}
    s_n(z)=\int\frac{1}{\lambda-z}dF^{W_n}(\lambda)=\frac{1}{2n}{\rm tr}(W_n-zI_{2n})^{-1}.
\end{align*}
By Lemma \ref{lemma:7},  we have:
\begin{equation}\label{eq:5}
\left| {{\rm tr}{{({W_n} - z{I_{2n}})}^{ - 1}} - {\rm tr}{{(W_n^{(k)} - z{I_{2n - 2}})}^{ - 1}}} \right| \le \frac{2}{\upsilon}.
\end{equation}
Note that here we use the fact that ${W_n^{(k)}}$ has two  rows and two columns fewer than $W_n$.
 Check that $\{ {\gamma _k}\} $  forms a sequence of bounded martingale differences, thus by Lemma \ref{lemma:8},  we obtain
$${\rm E}{\left| {{s_n}(z) -{\rm E}{s_n}(z)} \right|^4} \le \frac{{{K_4}}}{{{{(2n)}^4}}}{\rm E}
{\left( {\sum\limits_{k = 1}^n {{{\left| {{\gamma _k}} \right|}^2}} } \right)^2} \le \frac{{16{K_4}}}{{{n^2}{\upsilon^4}}}
= O({n^{ - 2}}). $$
which  together with Borel-Cantelli Lemma  implies that, for each fixed $ z\in \mathbb{C}^+ $,
$${s_n}(z) -{\rm E}{s_n}(z) \to 0 \quad a. s. $$

{\bf{ Step 2}}:
Denote  $Q_k=(x'_{1k}, \ldots, x'_{(k-1)k}, x'_{(k+1)k}, \ldots, x'_{nk})'$,
$$ R_k(j, l)=
\left(\begin{array}{cc}
e_{jl}(k) & g_{jl}(k)\\
h_{jl}(k) & f_{jl}(k)
\end{array}\right),
R_k=(W_n^{(k)}-zI_{2n-2})^{-1}=(R_k(j, l))_{j, l\ne k},
$$

$$R(j, l)=
\left(\begin{array}{cc}
e_{jl} & g_{jl}\\
h_{jl} & f_{jl}\end{array}\right)\mbox{ and }
R=(W_n-zI_{2n})^{-1}=(R(j, l)).$$
By Lemma \ref{lemma:9},  we have
$${s_n}(z) = \frac{1}{{2n}}{\rm tr}{({W_n} - z{I_{2n}})^{ - 1}}
= \frac{1}{{2n}}\sum\limits_{k = 1}^n {{\rm tr}( - z{I_2} - {Q^*}_k{{(W_n^{(k)} - z{I_{2n - 2}})}^{ - 1}}{Q_k}} {)^{ - 1}}.$$
Let $\varepsilon_k={\rm E}s_n(z)I_2-\frac{1}{n}Q_k^*R_kQ_k, $ then  we have $${\rm E}{s_n}(z) =  - \frac{1}{{z + {\rm E}{s_n}(z)}} + {\delta _n}$$
where $${\delta _n} = \frac{1}{{2n}}\sum\limits_{k = 1}^n {{\rm {Etr}}(\frac{1}{{z + {\rm E}{s_n}(z)}}} {\varepsilon _k}{(( - z - {\rm E}{s_n}(z)){I_2} + {\varepsilon _k})^{ - 1}}).$$
Solving for ${\rm E}{s_n}(z)$ from the equation above and according to the analysis in Page 36 of \cite{BaiS10S}, it is suffices to show that $${\delta _n} \to 0.$$

Now, rewrite
\begin{align}\label{eq7}
\delta_n&=-\frac{1}{2n(z+{\rm E}s_n(z))^2}\sum_{k=1}^{n}{\rm{Etr}}\varepsilon_k \notag \\
+&\frac{1}{2n(z+{\rm E}s_n(z))^2}\sum_{k=1}^{n}{\rm{Etr}}(\varepsilon_k^2(\varepsilon_k-zI_2-{\rm E}s_n(z)I_2)^{-1})\notag \\
&=J_1+J_2.
\end{align}

By (\ref{eq:5}), we have
\begin{align}\label{eq8}
\vert {\rm{Etr}}\varepsilon_k\vert & =\vert {\rm{tr}}({\rm E}s_n(z)I_2-\frac{1}{n}{\rm E}Q_k^*(W_n^{(k)}-zI_{2n-2})^{-1}Q_k)\vert \notag \\
& =\vert 2{\rm E}s_n(z)-\frac{1}{n}{\rm{Etr}}(W_n^{(k)}-zI_{2n-2})^{-1}\vert \notag\\
& =\frac{1}{n}\vert {\rm{Etr}}(W_n-zI_{2n})^{-1}-{\rm{Etr}}(W_n^{(k)}-zI_{2n-2})^{-1} \vert\leq \frac{2}{n\upsilon}.
\end{align}

Then, note that the Stieltjes transform sends $\mathbb{C}^+$ to $\mathbb{C}^+$, we obtain
\begin{equation}\label{eq9}
\vert z+s_n(z)\vert\geq \Im(z+s_n(z))=\upsilon+\Im(s_n(z))\geq \upsilon,
\end{equation}

By (\ref{eq8}) and (\ref{eq9}), we obtain, for any fixed $z\in\mathbb{C}^+$,
\begin{equation}\label{eq12}
\vert J_1\vert \leq \frac{1}{n\upsilon^3} \to 0.
\end{equation}

 Now we begin to prove $J_2\to 0$. Let $\alpha_k$ denote the first column of $Q_k$, and $\beta_k$ denote the second column of $Q_k$.
Write:
\begin{align}
(\varepsilon_k-zI_2-{\rm E}s_n(z)I_2)=
\left(\begin{array}{cc}
-z-\frac{1}{n}\alpha_k^*R_k\alpha_k & -\frac{1}{n}\alpha_k^*R_k\beta_k\\
-\frac{1}{n}\beta_k^*R_k\alpha_k & -z-\frac{1}{n}\beta_k^*R_k\beta_k
\end{array}\right) .
\end{align}
By Lemma \ref{lemma:1}, we obtain
\begin{align}\label{eq1}
&(\varepsilon_k-zI_2-{\rm E}s_n(z)I_2)^{-1}\notag\\
&=\left(\begin{array}{cc}
-z-\frac{1}{n}\alpha_k^*R_k\alpha_k & 0\\
0 & -z-\frac{1}{n}\beta_k^*R_k\beta_k
\end{array}\right)^{-1}
\end{align}
and
\begin{equation}\label{eq:6}
-z-\frac{1}{n}\alpha_k^*R_k\alpha_k=-z-\frac{1}{n}\beta_k^*R_k\beta_k.
\end{equation}
Using \eqref{eq1} and  (\ref{eq:6}), we have
\begin{align}\label{eq11}
{\rm E}\vert {\rm tr}\varepsilon_k^2\vert\notag &={\rm E}\vert [{\rm E}s_n(z)-\frac{1}{n}\alpha_k^*R_k\alpha_k]^2+[{\rm E}s_n(z)-\frac{1}{n}\beta_k^*R_k\beta_k]^2\vert \notag \\
&=2{\rm E}\vert {\rm E}s_n(z)-\frac{1}{n}\alpha_k^*R_k\alpha_k\vert^2 \notag \\
&=2{\rm E}\vert {\rm E}s_n(z)-\frac{1}{2n}(\alpha_k^*R_k\alpha_k+\beta_k^*R_k\beta_k)\vert^2 \notag \\
&=2{\rm E}\vert {\rm E}s_n(z)-\frac{1}{2n}{\rm tr}Q_k^*R_kQ_k\vert^2 \notag \\
&\leq8{\rm E}\{\vert\frac{1}{2n}[{\rm{tr}}Q_k^*R_kQ_k-{\rm{tr}}R_k]\vert^2+\vert\frac{1}{2n}[{\rm{tr}}R-{\rm{tr}}R_k]\vert^2+\vert s_n(z)-{\rm E}s_n(z)\vert^2\}.
\end{align}
What is more, from (\ref{eq:5}), we have
\begin{align}\label{eq3}
{\rm E}\vert\frac{1}{2n}&[{\rm{tr}}R-{\rm{tr}}R_k]\vert^2\leq\frac{1}{n^2\upsilon^2}.
\end{align}
By (\ref{eq:7}), (\ref{eq:8}), (\ref{eq:5}) and applying the  fact  that the martingale difference $\gamma_k$ are uncorrelated, for $k=1, \cdots, n$, we obtain
\begin{align}\label{eq4}
{\rm E}\vert s_n(z)-{\rm E}s_n(z)\vert^2&=\frac{1}{4n^2}\sum_{k=1}^{n}{\rm E}|\gamma_k|^2 \leq\frac{4}{n\upsilon^2}.
\end{align}
Now  considering the first term of \eqref{eq11}, we have
\begin{align}
&{\rm E}\vert {\rm tr}Q_k^*R_kQ_k-{\rm tr} R_k\vert^2\notag \\
&={\rm E}\vert \sum_{j\neq k}{}\sum_{l \neq k}{}{\rm tr}x_{lk}x_{jk}^*R_k(j, l)-\sum_{j \neq k}{}{\rm tr}R_k(j, j)\vert^2\notag \\
&={\rm E}\vert \sum_{j\neq k}{}\sum_{j \neq l, l \neq k}{}{\rm tr}x_{lk}x_{jk}^*R_k(j, l)+\sum_{j \neq k}{}{\rm tr}(x_{jk}x_{jk}^*-I_2)R_k(j, j)\vert^2\notag \\
&\leq 2 \sum_{j\neq k}{}\sum_{j \neq l, l \neq k}{} {\rm E}\vert {\rm tr}x_{lk}x_{jk}^*R_k(j, l)\vert^2+\sum_{j \neq k}{} {\rm E}\vert{\rm tr}(x_{jk}x_{jk}^*-I_2)R_k(j, j)\vert^2.\notag
\end{align}
Since for $l\ne j$,
\begin{align}
{\rm E}\vert {\rm tr} x_{lk}x_{jk}^*R_k(j, l)\vert^2&={\rm E}\vert \rm{tr}
\left(\begin{array}{cc}
\lambda_{lk} & \omega_{lk}\\
-\overline{\omega}_{lk} & \overline{\lambda}_{lk}
\end{array}\right)
\left(\begin{array}{cc}
\overline{\lambda}_{jk} & -\omega_{jk}\\
\overline{\omega}_{jk} &\lambda_{jk}
\end{array}\right)
\left(\begin{array}{cc}
e_{jl}(k) & g_{jl}(k)\\
h_{jl}(k) & f_{jl}(k)
\end{array}\right)\vert^2\notag \\
&={\rm E}\vert(\lambda_{lk}\overline{\lambda}_{jk}+\omega_{lk}\overline{\omega}_{jk})e_{jl}(k)+(\omega_{lk}\lambda_{jk}-\lambda_{lk}\omega_{jk})h_{jl}(k)\notag \\
&+(\overline{\lambda}_{lk}\overline{\omega}_{jk}-\overline{\lambda}_{jk}\overline{\omega}_{lk})g_{jl}(k)+(\omega_{jk}\overline{\omega}_{lk}+\overline{\lambda}_{lk}\lambda_{jk})f_{jl}(k)\vert^2\notag
\\&\leq4{\rm E}\{\vert\lambda_{lk}\overline{\lambda}_{jk}+\omega_{lk}\overline{\omega}_{jk}\vert^2\vert e_{jl}(k)\vert^2+\vert\omega_{lk}\lambda_{jk}-\lambda_{lk}\omega_{jk}\vert^2\vert h_{jl}(k)\vert^2\notag \\
&+\vert\overline{\lambda}_{lk}\overline{\omega}_{jk}-\overline{\lambda}_{jk}\overline{\omega}_{lk}\vert^2\vert g_{jl}(k)\vert^2+\vert\omega_{jk}\overline{\omega}_{lk}+\overline{\lambda}_{lk}\lambda_{jk}\vert^2\vert f_{jl}(k)\vert^2\}\notag \\
&\leq 4{\rm E}(\vert e_{jl}(k)\vert^2+\vert f_{jl}(k)\vert^2+\vert h_{jl}(k)\vert^2+\vert g_{jl}(k)\vert^2)\notag
\end{align}
and for $j=l$,
\begin{align}
&{\rm E}\vert {\rm tr} (x_{jk}x_{jk}^*-I_2)R_k(j, j)\vert^2\notag \\
&={\rm E}\vert (|\lambda_{jk}|^2+|\omega_{jk}|^2-1)\rm{tr}
\left(\begin{array}{cc}
f_{jj}(k) & 0\\
0& f_{jj}(k)
\end{array}\right)\vert^2\notag \\
&={\rm E}\vert (|\lambda_{jk}|^2+|\omega_{jk}|^2-1)(2f_{jj}(k))\vert^2
\notag \\
&=4{\rm E}\vert (|\lambda_{jk}|^2+|\omega_{jk}|^2-1)\vert^2\vert  f_{jj}(k)\vert^2
\notag \\
&\leq 4\eta_n^2n{\rm E}(\vert e_{jj}(k)\vert^2+\vert f_{jj}(k)\vert^2+\vert h_{jj}(k)\vert^2+\vert g_{jj}(k)\vert^2).\notag
\end{align}
Therefore,  for all large $n$, we have
\begin{align}
\mbox E\vert \mbox{tr}Q_k^*R_kQ_k-\mbox{tr}R_k\vert^2
&\leq 8\sum_{j\neq k}^{}\sum_{j\neq l, l\neq k}^{}\mbox E(\vert e_{jl}(k)\vert^2+\vert f_{jl}(k)\vert^2+\vert h_{jl}(k)\vert^2+\vert g_{jl}(k)\vert^2)\notag \\&+4\eta_n^2n\sum_{j\neq k}^{}\mbox E(\vert e_{jl}(k)\vert^2+\vert f_{jl}(k)\vert^2+\vert h_{jl}(k)\vert^2+\vert g_{jl}(k)\vert^2)\notag \\
&\leq 4n\eta_n^2{\rm Etr}(R_k R_k^*).\notag
\end{align}
By the  fact $$\mbox {Etr}(R_k R_k^*)=2(n-1)\int_{-\infty}^{+\infty}\frac{1}{|x-z|^2}d\mbox EF_n^{(k)}(x)\leq 2n\upsilon^{-2},$$ where $F_n^{(k)}(x)$ is the ESD of $R_k$, we have
\begin{equation}\label{eq5}
\frac{1}{4n^2}\mbox E\vert \mbox{tr}Q_k^*R_kQ_k-\mbox{tr}R_k\vert^2\leq \frac{2\eta_n^2}{\upsilon^2} \to 0.
\end{equation}
Combining  (\ref{eq11}), (\ref{eq3}), (\ref{eq4}), and (\ref{eq5}), we obtain, for all large $n$,
\begin{equation}\label{eq13}
\mbox E|\mbox{tr}\varepsilon_k^2| \to 0.
\end{equation}
By (\ref{eq1}), we have
\begin{align}\label{al:1}
(\varepsilon_k-zI_2-\mbox Es_n(z)I_2)^{-1}
&=\left(\begin{array}{cc}
-z-\frac{1}{n}\alpha_k^*R_k\alpha_k & 0\\
0 & -z-\frac{1}{n}\beta_k^*R_k\beta_k
\end{array}\right)^{-1}\notag \\
&=\frac{1}{-z-\frac{1}{n}\alpha_k^*R_k\alpha_k}I_2.
\end{align}
Thus from (\ref{eq13}) and (\ref{al:1})
 we have
 \begin{align}\label{eq10}
 &\mbox E|\mbox{tr}(\varepsilon_k^2(\varepsilon_k-zI_2-s_n(z)I_2)^{-1})|\notag\\=&
\mbox E|\frac{\mbox{tr}\varepsilon_k^2}{-z-\frac{1}{n}\alpha_k^*R_k\alpha_k}|
\le\frac{\mbox E|\mbox{tr}\varepsilon_k^2|}{\upsilon} \to 0.
\end{align}
Here we use the fact that
 \begin{equation}
  \begin{split}
\Im(-z-\frac{1}{n}\alpha_k^*R_k{\alpha_k})=- \upsilon\left( {1 + \frac{1}{n}\alpha _k^ * {{ {R_kR_k^*} }}{\alpha _k}} \right)
 <  - \upsilon.
 \end{split}
\end{equation}
Therefore, by (\ref{eq7}), (\ref{eq9}),  (\ref{eq12}),  and (\ref{eq10}), we conclude that, for all large $n$,
\begin{align}
|\delta_n|  \to 0,
\end{align}
which completes the proof of the mean convergence $\mbox E{s_n}(z) \to s(z)$. And the proof of Theorem \ref{th1} is complete.

\section{Appendix}
Let us make a list of lemmas that were used in the process of the proofs of Lemma \ref{lemma:1} and  Theorem \ref{th1}.
\begin{lemma}[Corollary A.41 in \cite{BaiS10S}]\label{lemma:4}
Let $A$ and $B$ be two $n \times n$  normal matrices with their ESDs ${F^A}$ and ${F^B}$.  Then, $${L^3}({F^A}, {F^B}) \le \frac{1}{n}{\rm tr}[(A - B){(A - B)^ * }],$$
where $L(\cdot, \cdot)$ is the Levy distance between two distributions (See Remark A.39. in \cite{BaiS10S}).
\end{lemma}

\begin{lemma}[Theorem A.43 in \cite{BaiS10S}]\label{lemma:3}
Let $A$ and $B$ be two $p \times n$ Hermitian matrices.  Then,  $$\left\| {{F^A} - {F^B}} \right\| \le \frac{1}{n}{\rm rank}(A - B). $$
\end{lemma}

\begin{lemma}[See appendix A.1.4 in \cite{BaiS10S}]\label{lemma:2}
Suppose that the matrix $\Sigma $ is nonsingular and has the partition as given by $\left( {\begin{array}{*{20}{c}}
{{\Sigma _{11}}}&{{\Sigma _{12}}}\\
{{\Sigma _{21}}}&{{\Sigma _{22}}}
\end{array}} \right),$ then, if $\Sigma_{11}$ is  nonsingular, the inverse of $\Sigma $ has the form $$\left( {\begin{array}{*{20}{c}}
{\Sigma _{11}^{ - 1} + \Sigma _{11}^{ - 1}\Sigma _{12}^{}\Sigma _{22. 1}^{ - 1}\Sigma _{21}^{}\Sigma _{11}^{ - 1}}&{ - \Sigma _{11}^{ - 1}\Sigma _{12}^{}\Sigma _{22. 1}^{ - 1}}\\
{ - \Sigma _{22. 1}^{ - 1}\Sigma _{21}^{}\Sigma _{11}^{ - 1}}&{\Sigma _{22. 1}^{ - 1}}
\end{array}} \right)$$
where $\Sigma _{22. 1}= \Sigma _{22}^{} - \Sigma _{21}^{}\Sigma _{11}^{ - 1}\Sigma _{12}^{}$.
\end{lemma}

\begin{lemma}[Theorem A.4 in \cite{BaiS10S}]\label{lemma:9}
For an $n\times n$ Hermitian $A$, define $A_k$, called a major submatrix of order $n-1$, to be the matrix resulting from the $k$-th row and column from $A$.
If both $A$ and $A_k$, $k=1, \cdots, n$,  are nonsigular, and if we write $A^{-1}=[a^{kl}]$, then
\begin{displaymath}
a^{kk}=\frac{1}{a_{kk}-\alpha_k^*A_k^{-1}\beta_k}
\end{displaymath}
and hence
\begin{displaymath}
{\rm tr}(A^{-1})=\sum_{k=1}^{n}\frac{1}{a_{kk}-\alpha_k^*A_k^{-1}\beta_k},
\end{displaymath}
where $a_{kk}$ is the $k$-th diagonal entry of $A$, $\alpha_k^{\prime}$ is the vector obtained from the $k$-th row of $A$ by deleting the $k$-th entry, and $\beta_k$ is the vector from the $k$-th column by deleting the $k$-th entry.
\end{lemma}

\begin{lemma}[See appendix A.1.5 in \cite{BaiS10S}]\label{lemma:7}
Let $z = u + iv, v > 0, $ and let $A$ be an $n \times n$ Hermitian matrix.  ${A_k}$ be the k-th major sub-matrix of $A$ of order $(n-1)$, to be the matrix resulting from the $k$-th row and column from $A$.  Then
$$\left| {{\rm tr}{{(A - z{I_n})}^{ - 1}} -{ \rm tr}{{({A_k} - z{I_{n - 1}})}^{ - 1}}} \right| \le \frac{1}{v}.$$
\end{lemma}

\begin{lemma}[Lemma 2.11 in \cite{BaiS10S}]\label{lemma:6}
Let $z = u + iv, v > 0, $ $s(z)$ be the Stieltjes transform of the semicircular law.  Then,  we have $s(z) =  - \frac{1}{2}(z - \sqrt {{z^2} - 4} )$.
\end{lemma}

\begin{lemma}[Lemma 2.12 in \cite{BaiS10S}]\label{lemma:8}
Let $\{ {\tau_k}\} $ be a complex martingale difference sequence with respect to the increasing $\sigma  - field$  .  Then ,  for $p > 1, $ ${\rm E}{\left| {\sum {{\tau_k}} } \right|^p} \le {K_p}{\rm E}{({\sum {\left| {{\tau_k}} \right|} ^2})^{p/2}}.$
\end{lemma}

\begin{lemma}[Page 29 in \cite{BaiS10S}]\label{lemma:5}
Let $ \tau_j $ are independent with zero means,   then we have,  for some constant $C_k$,
\begin{displaymath}
{\rm E}|\sum \tau_j|^{2k} \leq C_k(\sum{\rm E}|\tau_j|^{2k}+(\sum {\rm E}|\tau_j|^2)^k).
\end{displaymath}
\end{lemma}

%

\end{document}